\documentclass[reqno]{amsart}
\usepackage{amssymb, amsthm, amsmath, amsfonts}
\usepackage{graphics}
\usepackage{hyperref}
\usepackage[all]{xy}
\usepackage{enumerate}
\usepackage[mathscr]{eucal}
\usepackage{enumerate}
\usepackage{thmtools}

\declaretheorem{theorem}
\declaretheorem[name=Lemma, sibling=theorem]{lemma}
\declaretheorem[name=Proposition, sibling=theorem]{proposition}
\declaretheorem[name=Corollary, sibling=theorem]{corollary}

\declaretheorem[name=Remark, style=definition, sibling=theorem]{remark}

\DeclareMathOperator{\HH}{H}

\newcommand{\kk}{{\mathbf{k}}}
\newcommand{\HHH}{{\mathbf{H}}}
\newcommand{\II}{{\mathcal{I}}}
\newcommand{\FI}{{\mathbf{FI}}}
\newcommand{\FB}{{\mathbf{FB}}}
\newcommand{\GL}{{\mathrm{GL}}}
\newcommand{\ttt}{{\mathbf{t}}}

\title{Linear stable range for homology of congruence subgroups via FI-modules}

\author{Wee Liang Gan}
\address{Department of Mathematics, University of California, Riverside, CA 92521, USA}
\email{wlgan@ucr.edu}

\author{Liping Li}
\address{LCSM (Ministry of Education), School of Mathematics and Statistics, Hunan Normal University, Changsha, Hunan 410081, China}
\email{lipingli@hunnu.edu.cn}

\thanks{L. Li is supported by the National Natural Science Foundation of China (Grant No. 11771135) and the Research Foundation of Education Bureau of Hunan Province, China (Grant No. 18A016). The authors thank the anonymous referee for carefully checking the manuscript and providing many valuable comments.}

\begin{document}

\begin{abstract}
We answer positively a question of Church, Miller, Nagpal and Reinhold on existence of a linear bound on the presentation degree of the homology of a complex of FI-modules. This implies a linear stable range for the homology of congruence subgroups of general linear groups.
\end{abstract}

\maketitle

\section*{Introduction}

Let $R$ be a ring. For any finite set $S$, let $R^S$ be the free right $R$-module on $S$ and let $\GL_S(R)$ be the group of right $R$-module automorphisms of $R^S$. For any two-sided ideal $I$ of $R$, the congruence subgroup $\GL_S(R, I)$ is defined to be the kernel of the natural group homomorphism $\GL_S(R) \to \GL_S(R/I)$. When $S$ is the finite set $[n]:=\{1,\ldots,n\}$, we denote $\GL_S(R, I)$ by $\GL_n(R,I)$.

Suppose $T$ is a finite set and $S\subset T$. The natural group monomorphism $\GL_S(R) \to \GL_T(R)$ restricts to a group monomorphism $\GL_S(R, I) \to \GL_T(R, I)$. Taking group homology with coefficients in an abelian group $A$, we get a group homomorphism $\HH_k(\GL_S(R,I); A) \to \HH_k(\GL_T(R,I); A)$. For each $n\geqslant 0$ and $N \geqslant 0$, there is a canonical homomorphism:
\[  \mathop{\mathrm{colim}}_{\substack{S\subset [n]\\ |S|\leqslant N }} \HH_k(\GL_S(R, I) ;A)
\longrightarrow \HH_k (\GL_n(R, I); A) \]
where the colimit is taken over the poset of all subsets $S$ of $[n]$ such that $|S|\leqslant N$.

In this article, we prove the following result:

\begin{theorem} \label{central stability}
Suppose that $R$ satisfies Bass's stable range condition $\mathrm{SR}_{d+2}$ for some $d\geqslant 0$ and $I$ is a proper two-sided ideal of $R$. Then for each $k\geqslant 0$ and $n\geqslant 0$, one has a canonical isomorphism:
\begin{equation} \label{colimit}
\mathop{\mathrm{colim}}_{\substack{S\subset [n]\\ |S|\leqslant \omega(k) }} \HH_k(\GL_S(R, I) ;A) \stackrel{\sim}{\longrightarrow}  \HH_k(\GL_n(R,I); A)
\end{equation}
where $\omega(k) = 4k+2d+6$.
\end{theorem}

We refer the reader to \cite[(V, 3.1)]{bass} for the definition of condition $\mathrm{SR}_{d+2}$. Bass proved in \cite[(V, 3.5)]{bass} that any commutative Noetherian ring of Krull dimension $d$ satisfies $\mathrm{SR}_{d+2}$. For example, any field satisfies $\mathrm{SR}_2$ and any Dedekind domain satisfies $\mathrm{SR}_3$.

It is trivial that \eqref{colimit} is an isomorphism when $n\leqslant \omega(k)$; the range $n > \omega(k)$ is called the \emph{stable range}. Qualitatively, Theorem \ref{central stability} says that there is a stable range starting at a linear function of $k$.

An excellent account of what was known about homology of the congruence subgroups is given in \cite{putman}. Let us recall the recent developments most relevant to our present article\footnote{See also \cite{calegari}, \cite{djament}, \cite{mpw}, \cite{ps} for recent related work.}:
\begin{enumerate}[(i)]
\item
Putman proved Theorem \ref{central stability} with stable range $n> (d+8)2^{k-1}-4$ under the assumption that $A$ is a field whose characteristic is either $0$ or at least $(d+8) 2^{k-1} - 3$; see \cite[Theorem B]{putman}.

\item
Church-Ellenberg-Farb-Nagpal proved Theorem \ref{central stability} without an explicit stable range in the special case where $R$ is the ring of integers of a number field and under the assumption that $A$ is a Noetherian ring; see \cite[Theorems C and D]{cefn}.

\item
Church-Ellenberg proved Theorem \ref{central stability} with stable range $n> 2^{k-2}(2d+9)-1$; see \cite[Theorem D']{ce}.

\item
Church-Miller-Nagpal-Reinhold proved Theorem \ref{central stability} with stable range $n> 4k^2 + (4d+10)k + 5d + 6$; see \cite[Application B]{cmnr}.

\end{enumerate}

Most importantly for us, Church-Miller-Nagpal-Reinhold reduced the proof of Theorem \ref{central stability} to a question on $\FI$-modules which may be of independent interest; see \cite[Question 5.2]{cmnr}. We give a positive answer to their question in Theorem \ref{main theorem}. Recently, Jeremy Miller and Jennifer Wilson applied our answer to this question to deduce a result similar to Theorem \ref{central stability} for cohomology groups of ordered configuration spaces; for details, see \cite{mw}.

This article is organized as follows. In Section \ref{generalities}, we recall the definitions and results on $\FI$-modules that we need. In Section \ref{main result}, we state and prove Theorem \ref{main theorem}. In Section \ref{application}, we give the application of Theorem \ref{main theorem} to the proof of Theorem \ref{central stability}.

We thank Thomas Church and Jeremy Miller for answering our questions on \cite{cmnr}. Jeremy Miller and Rohit Nagpal informed us that they have recently proved Theorem \ref{central stability} with a stable range starting at a linear function of $k$ in the special case where $R$ is the ring of integers of a number field. Their method of proof is completely different from the one we give here.

\section{Generalities} \label{generalities}

We refer the reader to \cite{ce} for background on $\FI$-modules; our notations will follow that of \cite{cmnr}.

We work over a commutative ring $\kk$. Let $\FI$ be the category of finite sets and injective maps. By an $\FI$-module (respectively, $\FI$-group), we mean a functor from $\FI$ to the category of $\kk$-modules (respectively, groups); see \cite{cef}. Let $\FB$ be the category of finite sets and bijective maps. An $\FB$-module is a functor from $\FB$ to the category of $\kk$-modules. Suppose $M$ is an $\FI$-module or $\FB$-module. We write $M_S$ for the value of $M$ on a finite set $S$. We write $M_n$ for the value of $M$ on $[n]$. If $M$ is nonzero, its degree $\deg M$ is defined to be $\sup\{ n \mid M_n\neq 0\}$; if $M$ is zero, we set $\deg M$ to be $-1$.

For any $\FB$-module $V$, the induced $\FI$-module $\II(V)$ is defined by
\[ \II(V)_S = \bigoplus_{n\geqslant 0} \kk[ \mathrm{Hom}_{\FI}([n], S) ] \otimes_{\kk[S_n]} V_n, \]
where $S_n$ denotes the symmetric group on $[n]$.  If $\alpha: S\to T$ is an injective map between finite sets, then $\alpha_*: \II(V)_S \to \II(V)_T$ is the $\kk$-linear map defined by
\[ \beta \otimes v \mapsto (\alpha\circ \beta) \otimes v \]
where $\beta\in \mathrm{Hom}_{\FI}([n], S)$ and $v \in V_n$ for any integer $n\geqslant 0$.
The functor $V \mapsto \II(V)$ is a left adjoint functor to the forgetful functor from the category of $\FI$-modules to the category of $\FB$-modules.
The projective $\FI$-modules are the induced $\FI$-modules $\II(V)$ where each $V_n$ is projective as a $\kk[S_n]$-module.

Suppose $V$ is an $\FB$-module. Then each $V_n$ can be regarded as an $\FB$-module whose value on a finite set $S$ is $V_S$ if $|S|=n$, zero if $|S|\neq n$. There is a direct sum decomposition $V=\bigoplus_{n\geqslant 0} V_n$ in the category of $\FB$-mnodules. Correspondingly, there is a direct sum decomposition
\[\II(V) = \bigoplus_{n\geqslant 0} \II(V_n)\]
in the category of $\FI$-modules.

Suppose $V$ and $W$ are $\FB$-modules. Since there are $\FI$-module direct sum decompositions
\[ \II(V)= \bigoplus_{n\geqslant 0} \II(V_n)\quad \mbox{ and }\quad \II(W)=\bigoplus_{m\geqslant 0} \II(W_m),\]
any $\FI$-module homomorphism $f: \II(V) \to \II(W)$ is given by a matrix $(f^{m,n})$ of $\FI$-module homomorphisms
\[ f^{m,n} : \II(V_n) \to \II(W_m). \]
If $m>n$, then $\II(W_m)_n=0$, so there is no nonzero homomorphism from $\II(V_n)$ to $\II(W_m)$. Thus, the matrix $(f^{m,n})$ is upper-triangular.

For any $\FI$-module $M$ and finite set $S$, we write $M_{\prec S}$ for the $\FI$-submodule of $M$ generated by all $M_T$ with $|T|<|S|$. Define an $\FI$-module $\HH_0^{\FI}(M)$ by
\[ \HH_0^{\FI}(M)_S  =  (M/M_{\prec S})_S. \]
The functor $\HH_0^{\FI}$ is right exact; let $\HH_i^{\FI}$ be its $i$-th left derived functor. For any $\FI$-module $M$, denote by $t_i(M)$ the degree $\deg \HH_i^\FI (M)$. We call $\max\{ t_0(M), t_1(M) \}$ the \emph{presentation degree} of $M$.

If $V$ is an $\FB$-module, then it is plain that $\HH_0^\FI(\II(V)) = V$ and $t_0(\II(V))= \deg V$.

The following theorem is due to Church-Ellenberg \cite[Theorem C]{ce}; see \cite{gl} for a very simple proof.

\begin{theorem}[Church-Ellenberg]  \label{presentation}
Let $M$ be an $\FI$-module. If the presentation degree of $M$ is at most $N$, then there is a canonical isomorphism
\[\mathop{\mathrm{colim}}_{\substack{S\subset [n]\\ |S|\leqslant N }} M_S
\stackrel{\sim}{\longrightarrow}  M_n \quad \mbox{ for every } n\geqslant 0. \]
\end{theorem}

The next theorem is also due to Church-Ellenberg \cite[Theorem A]{ce}; alternative proofs can be found in \cite{church}, \cite{gan} and \cite{li}.

\begin{theorem}[Church-Ellenberg] \label{regularity bound}
Let $M$ be an $\FI$-module. Then for each $i\geqslant 2$, one has:
\[ t_i(M) \leqslant t_0(M) + t_1(M) + i - 1. \]
\end{theorem}

From Theorem \ref{regularity bound}, it is easy to deduce the following result.

\begin{corollary} \label{key lemma}
Let $P$ and $Q$ be projective $\FI$-modules. Let $f : P \to Q$ be a homomorphism and $Z = \ker(f)$. Then for each $i\geqslant 0$, one has:
\[ t_i(Z) \leqslant 2t_0(P) + i + 1. \]
\end{corollary}
\begin{proof}
The claim is trivial if $t_0(P)=\infty$, so suppose $t_0(P)<\infty$. Let $N=t_0(P)$.

There are  $\FB$-modules $V$ and $W$ such that $P=\II(V)=\bigoplus_{n\geqslant 0} \II(V_n)$ and $Q = \II(W)= \bigoplus_{m\geqslant 0} \II(W_m)$. Since the matrix $(f^{m,n})$ is upper-triangular and $V_n=0$ for every $n>N$, the image of $f$ is contained in $\bigoplus_{m \leqslant N} \II(W_m)$. Therefore, replacing $Q$ by $\bigoplus_{m\leqslant N} \II(W_m)$, we may assume that $t_0(Q)\leqslant N$.

From the short exact sequence $0 \to Z \to P \to P/Z \to 0$, we deduce that:
\begin{gather*}
t_0(Z) \leqslant \max\{ t_0(P), t_1(P/Z) \} ,\\
t_i(Z) = t_{i+1}(P/Z) \;\mbox{ if } i\geqslant 1.
\end{gather*}
Let $C$ be the cokernel of $f$. From the short exact sequence $0\to P/Z \to Q \to C \to 0$, we deduce that $t_{i+1}(P/Z) = t_{i+2}(C)$; moreover, $t_0(C)\leqslant t_0(Q)$ and $t_1(C)\leqslant t_0(P/Z) \leqslant t_0(P)$. Applying Theorem \ref{regularity bound}, we have:
\[ t_{i+2}(C) \leqslant t_0(C)+t_1(C) + i+1 \leqslant t_0(Q) + t_0(P) + i+1 \leqslant 2N+i+1. \]
Putting the above inequalities together gives the corollary.
\end{proof}

We write $\HHH_k^\FI$ for the $k$-th left hyper-derived functor of $\HH_0^\FI$. If $M_\bullet$ is a bounded-below chain complex of $\FI$-modules and $P_\bullet$ is a complex of projective (or induced\footnote{Induced $\FI$-modules are $\FI$-homology acyclic; see \cite[Theorem 1.3]{ly} or \cite[Theorem B]{ramos}.}) $\FI$-modules quasi-isomorphic to $M_\bullet$, then $\HHH_k^\FI (M_\bullet)$ is, by definition, equal to the $k$-th homology of the chain complex $ \HH_0^\FI(P_\bullet) $; denote by $\ttt_k(M_\bullet)$ the degree $\deg \HHH_k^\FI (M_\bullet)$.

\section{Bounding presentation degree of homology} \label{main result}

Let $M_\bullet$ be a complex of $\FI$-modules supported on non-negative homological degrees; its $k$-th homology $\HH_k(M_\bullet)$ is an $\FI$-module. The following theorem gives a positive answer to \cite[Question 5.2]{cmnr}.

\begin{theorem} \label{main theorem}
For each $k\geqslant 0$, one has:
\begin{align*}
t_0( \HH_k(M_\bullet ) ) &\leqslant 2 \ttt_k(M_\bullet) + 1,\\
t_1( \HH_k(M_\bullet )) &\leqslant  2 \max\{ \ttt_k(M_\bullet), \ttt_{k+1}(M_\bullet) \} + 2.
\end{align*}
In particular, if $a\geqslant 0$ and $\ttt_k(M_\bullet) \leqslant ak+b$ for every $k$, then one has:
\begin{align*}
t_0( \HH_k(M_\bullet ) ) &\leqslant 2ak + 2b + 1,\\
t_1( \HH_k(M_\bullet )) &\leqslant  2ak + 2a + 2b + 2.
\end{align*}
\end{theorem}

To prove Theorem \ref{main theorem}, let $P_\bullet$ be the total complex of a projective Cartan-Eilenberg resolution of $M_\bullet$. Then $P_\bullet$ is a chain complex of projective $\FI$-modules supported on non-negative degrees with $\HH_k(P_\bullet) = \HH_k(M_\bullet)$ for every $k$. Recall that, by definition, $\ttt_k(M_\bullet)$ is the degree of the $k$-th homology of the chain complex $\HH_0^{\FI}(P_\bullet)$.

For each $k$, there is an $\FB$-module $V_k$ such that
\[ P_k = \II(V_k) = \bigoplus_{n\geqslant 0} \II((V_k)_n) , \]
where $(V_k)_n$ is the value of $V_k$ at the object $[n]$. To avoid confusion in notation, we shall always reserve $k$ for homological degree and $n$ for the order of the finite set $[n]$.

We write $d$ for the differential map of the chain complex $P_\bullet$. Then $d: P_k \to P_{k-1}$ is an upper-triangular matrix $(d^{m,n})$ of homomorphisms
\[ d^{m,n} : \II( (V_k)_n ) \to \II( (V_{k-1})_m ). \]
In particular, since $( \II( (V_k)_n ) )_n = (V_k)_n$ and $( \II( (V_{k-1})_n ) )_n = (V_{k-1})_n$, the homomorphism $d^{n,n}$ at $[n]$ is a map $(V_k)_n \to (V_{k-1})_n$. One has $\HH_0^\FI (P_k) = V_k$. Thus, $\HH_0^\FI (P_\bullet)$ is the chain complex $V_\bullet$ whose differential map $\bar{d} : V_k \to V_{k-1}$ is defined at $[n]$ to be the map $d^{n,n}$ at $[n]$. One has:
\[ \ttt_k(M_\bullet) = \deg( \HH_k(V_\bullet) ). \]

\begin{lemma} \label{exactness}
Suppose $N> \ttt_k(M_\bullet)$. Then the sequence
\[ \xymatrix{
\II((V_{k+1})_N) \ar[r]^{d^{N,N}} & \II( (V_k)_N) \ar[r]^{d^{N,N}} &  \II( (V_{k-1})_N) }  \]
is exact.
\end{lemma}
\begin{proof}
Since $N>\ttt_k(M_\bullet)$, the sequence
\[ (V_{k-1})_N \stackrel{\bar{d}}{\longrightarrow} (V_k)_N \stackrel{\bar{d}}{\longrightarrow} (V_{k-1})_N \]
is exact. Since $\bar{d}$ at $[N]$ is precisely $d^{N,N}$ at $[N]$, and for every finite set $S$, $\kk [\mathrm{Hom}_{\FI} ([N], S)] $ is a free right $\kk[S_N]$-module, the lemma follows.
\end{proof}

Let $Z_k$ be the kernel of $d: P_k \to P_{k-1}$, and let $B_k$ be the image of $d: P_{k+1} \to P_k$. Thus, $\HH_k( P_\bullet) = Z_k / B_k$. Let
\[  f: \bigoplus_{n\leqslant \ttt_k(M_\bullet)} \II( (V_k)_n ) \to P_{k-1} \]
be the restriction of $d: P_k \to P_{k-1}$ to $\bigoplus_{n\leqslant \ttt_k(M_\bullet)} \II( (V_k)_n )$. Let
\[ Z = \ker(f).\]

\begin{lemma} \label{generators}
The following composition is surjective:
\[ Z \hookrightarrow Z_k \twoheadrightarrow \HH_k(P_\bullet).  \]
\end{lemma}
\begin{proof}
We need to prove that for every finite set $S$ and $x\in (Z_k)_S$, there exists $y\in (B_k)_S$ such that $x-y \in Z_S$. Let $N$ be an integer such that
\[ x\in \bigoplus_{n\leqslant N} \II( (V_k)_n )_S. \]
If $N\leqslant \ttt_k(M_\bullet)$, then $x\in Z(S)$ and we are done. Suppose that $N> \ttt_k(M_\bullet)$. We write
$x= x' + x''$ where
\[ x' \in \bigoplus_{n\leqslant N-1} \II( (V_k)_n )_S, \qquad x'' \in \II( (V_k)_N)_S.\]
Since the matrix $(d^{m,n})$ is upper-triangular, we have:
\[ d(x') \in \bigoplus_{n\leqslant N-1} \II( (V_{k-1})_n )_S, \qquad d(x'') \in \bigoplus_{n\leqslant N} \II( (V_{k-1})_n )_S. \]
But $d(x') + d(x'')= d(x) = 0$, so we must have:
\[ d(x'') \in \bigoplus_{n\leqslant N-1} \II( (V_{k-1})_n )_S. \]
Hence, $d^{N, N} (x'') = 0$. Since $N>\ttt_k(M_\bullet)$, by Lemma \ref{exactness}, there exists $w\in \II( (V_{k+1})_N)_S$ such that $d^{N,N}(w) = x''$. Since we also have
\[ d(w) \in \bigoplus_{n\leqslant N} \II( (V_k)_n )_S ,\]
it follows that
\[ x - d(w) \in \bigoplus_{n\leqslant N-1} \II( (V_k)_n )_S. \]
If $N-1 \leqslant \ttt_k(M_\bullet)$, then we are done. If not, we repeat the above argument with $x$ replaced by $x-d(w)$.
\end{proof}

Next, let
\[ \widetilde{f} : \bigoplus_{n\leqslant \max\{\ttt_k(M_\bullet), \ttt_{k+1}(M_\bullet) \} } \II ( (V_k)_n) \to P_{k-1} \]
be the restriction of $d: P_k \to P_{k-1}$ to $\bigoplus_{n\leqslant \max\{\ttt_k(M_\bullet), \ttt_{k+1}(M_\bullet)\} } \II ( (V_k)_n) $. Let
\[ \widetilde{Z} = \ker(\widetilde{f}).\]
Then $Z\subset \widetilde{Z}$, so by Lemma \ref{generators}, the composition
\[ \widetilde{Z} \hookrightarrow Z_k \twoheadrightarrow \HH_k(P_\bullet) \]
is surjective. Since the kernel of $Z_k \twoheadrightarrow \HH_k(P_\bullet)$ is $B_k$, the kernel of
$\widetilde{Z} \twoheadrightarrow \HH_k(P_\bullet)$ is $\widetilde{Z}\cap B_k$.

\begin{lemma} \label{kernel}
One has: $t_0(\widetilde{Z}\cap B_k) \leqslant \max\{ \ttt_k(M_\bullet), \ttt_{k+1}(M_\bullet) \}$.
\end{lemma}
\begin{proof}
Let
\[  g : \bigoplus_{n\leqslant \max\{ \ttt_k(M_\bullet), \ttt_{k+1}(M_\bullet) \}} \II((V_{k+1})_n) \to P_k \]
be the restriction of $d: P_{k+1} \to P_k$ to $\bigoplus_{n\leqslant \max\{ \ttt_k(M_\bullet), \ttt_{k+1}(M_\bullet) \}} \II((V_{k+1})_n)$. Let $\widetilde{B}$ be the image of $g$. We claim that $ \widetilde{Z}\cap B_k = \widetilde{B}$, which would prove the lemma. It is clear that $\widetilde{B} \subset \widetilde{Z}\cap B_k $. We need to prove that $\widetilde{Z}\cap B_k \subset \widetilde{B}$.

Suppose $S$ is a finite set and $x\in (\widetilde{Z}\cap B_k)_S$. Then there exists an integer $N$ and
\[ y \in \bigoplus_{n\leqslant N} \II( (V_{k+1})_n )_S \quad \mbox{ such that }\quad  d(y) =x.\]
If $N\leqslant \max\{ \ttt_k(M_\bullet), \ttt_{k+1}(M_\bullet) \}$, then we are done. Suppose that $N > \max\{ \ttt_k(M_\bullet), \ttt_{k+1}(M_\bullet) \}$. We write $y= y' + y''$ where
\[ y' \in \bigoplus_{n\leqslant N-1} \II( (V_{k+1})_n )_S , \qquad y'' \in \II( (V_{k+1})_N )_S. \]
Then
\[ d(y')\in \bigoplus_{n\leqslant N-1} \II( (V_{k})_n )_S, \qquad  d(y'')\in \bigoplus_{n\leqslant N} \II( (V_{k})_n )_S. \]
But
\[ d(y) = x \in  \widetilde{Z}_S \subset \bigoplus_{n\leqslant \max\{ \ttt_k(M_\bullet), \ttt_{k+1}(M_\bullet) \}} \II( (V_{k})_n )_S. \]
Hence, $d^{N,N} (y'') = 0$. Since $N>\ttt_{k+1}(M_\bullet)$,  by Lemma \ref{exactness}, there exists $w\in \II( (V_{k+2})_N )_S$ such that $d^{N,N}(w)= y''$.
Since we also have
\[ d(w) \in \bigoplus_{n\leqslant N} \II( (V_{k+1})_n )_S, \]
it follows that
\[ y-d(w) \in  \bigoplus_{n\leqslant N-1} \II( (V_{k+1})_n )_S .\]
One has $d(y-d(w)) = d(y) = x$. If $N-1 \leqslant \max\{ \ttt_k(M_\bullet), \ttt_{k+1}(M_\bullet) \}$, then we are done. If not, we repeat the above argument with $y$ replaced by $y-d(w)$.
\end{proof}

We now prove Theorem \ref{main theorem}.

\begin{proof}[Proof of Theorem \ref{main theorem}]
By Lemma \ref{generators}, we have $t_0(\HH_k(P_\bullet)) \leqslant t_0(Z)$. By Corollary \ref{key lemma}, we have $t_0(Z) \leqslant 2 \ttt_k(M_\bullet) + 1$. Hence,
\[ t_0(\HH_k(P_\bullet)) \leqslant 2 \ttt_k(M_\bullet) + 1. \]

Recall the short exact sequence:
\[ 0 \to \widetilde{Z}\cap B_k \to \widetilde{Z} \to \HH_k(P_\bullet)\to 0. \]
Lemma \ref{kernel} says that $t_0(\widetilde{Z}\cap B_k) \leqslant  \max\{ \ttt_k(M_\bullet), \ttt_{k+1}(M_\bullet) \}$. By Corollary \ref{key lemma}, we have
\[ t_1 (\widetilde{Z}) \leqslant 2 \max\{ \ttt_k(M_\bullet), \ttt_{k+1}(M_\bullet) \}  + 2. \]
Hence,
\[ t_1( \HH_k(P_\bullet)) \leqslant2 \max\{ \ttt_k(M_\bullet), \ttt_{k+1}(M_\bullet) \}  + 2. \]
\end{proof}

\begin{remark}
For any $\FI$-module $L$, denote by $\delta(L)$ the stable degree of $L$; see \cite[Definition 2.9]{cmnr}. It was proved in \cite[Theorem 5.1]{cmnr} that for each $k\geqslant 0$, one has:
\begin{equation*}
\delta( \HH_k(M_\bullet) ) \leqslant \ttt_k(M_\bullet).
\end{equation*}
Let us give an alternative proof of this inequality. Lemma \ref{generators} implies that $\HH_k(M_\bullet)$ is a subquotient of $\bigoplus_{n\leqslant \ttt_k(M_\bullet)} \II( (V_k)_n )$. Therefore, by \cite[Proposition 2.10]{cmnr}, one has:
\[ \delta( \HH_k(M_\bullet) )  \leqslant \delta( \bigoplus_{n\leqslant \ttt_k(M_\bullet)} \II( (V_k)_n ) ) = t_0 ( \bigoplus_{n\leqslant \ttt_k(M_\bullet)} \II( (V_k)_n ) ) \leqslant \ttt_k(M_\bullet). \]
\end{remark}

\section{Homology of congruence subgroups} \label{application}

We now work over $\mathbb{Z}$, so by $\FI$-modules we mean functors from $\FI$ to the category of $\mathbb{Z}$-modules. Let $R$ be a ring, let $I$ be a two-sided ideal of $R$, and denote by $\GL(R, I)$ the $\FI$-group $S\mapsto \GL_S(R, I)$.  Let $A$ be an abelian group.  Then $\HH_k(\GL(R, I); A)$ is an $\FI$-module whose value on a finite set $S$ is $\HH_k(\GL_S(R, I); A)$.

For any group $G$, let $E_\bullet G$ be the bar resolution of the trivial $\mathbb{Z} G$-module $\mathbb{Z}$ and let $C_\bullet(G; A) = E_\bullet G \otimes_G A$; so one has $\HH_k(C_\bullet(G; A)) = \HH_k(G; A)$. Then $C_\bullet(\GL(R, I); A)$ is a complex of $\FI$-modules such that
\[ \HH_k( C_\bullet (\GL(R, I); A ) ) = \HH_k(\GL(R, I); A). \]

We recall \cite[Proposition 5.4]{cmnr}:

\begin{proposition}[Church-Miller-Nagpal-Reinhold] \label{bass}
Suppose that $R$ satisfies Bass's stable range condition $\mathrm{SR}_{d+2}$ for some $d\geqslant 0$ and $I$ is a proper two-sided ideal of $R$. Then for each $k\geqslant 0$, one has:
\[ \ttt_k(C_\bullet (\GL(R, I); A )) \leqslant 2k+d. \]
\end{proposition}

We deduce that:

\begin{theorem} \label{presentation degree}
Suppose that $R$ satisfies Bass's stable range condition $\mathrm{SR}_{d+2}$ for some $d\geqslant 0$ and $I$ is a proper two-sided ideal of $R$. Then for each $k\geqslant 0$, one has:
\begin{gather*}
t_0(  \HH_k( \GL(R, I); A)  ) \leqslant  4k + 2d + 1,\\
t_1(  \HH_k( \GL(R, I); A)  ) \leqslant  4k + 2d + 6.
\end{gather*}
In particular, the presentation degree of $\HH_k( \GL(R, I); A)$ is at most $4k + 2d + 6$.
\end{theorem}
\begin{proof}
Immediate from Proposition \ref{bass} and Theorem \ref{main theorem}.
\end{proof}

Theorem \ref{central stability} now follows easily:

\begin{proof}[Proof of Theorem \ref{central stability}]
Immediate from Theorem \ref{presentation degree} and Theorem \ref{presentation}.
\end{proof}

The following corollary strengthens \cite[Theorem 1.4]{cefn}.

\begin{corollary} \label{polynomial growth}
Let $R$ be the ring of integers of a number field. Let $I$ be a proper ideal of $R$. For any $k\geqslant 0$ and any field $\kk$, there exists a polynomial $P(T)\in \mathbb{Q}[T]$ such that
\[ \dim_\kk \HH_k( \GL_n(R,I); \kk) = P(n) \quad \mbox{ for every } n > 8k+10. \]
\end{corollary}
\begin{proof}
We work over the field $\kk$. Since $R$ satisfies condition $\mathrm{SR}_3$, by Theorem \ref{presentation degree}, we have\footnote{The forgetful functor $\mathcal{U}$ from the category of $\FI$-modules over $\kk$ to the category of $\FI$-modules over $\mathbb{Z}$ is exact, sends induced modules to induced modules, and commutes with $\HH_0^{\FI}$. Thus, $\mathcal{U}$ commutes with $\HH_i^{\FI}$ for every $i\geqslant 0$; in particular, if $M$ is an $\FI$-module over $\kk$, then $t_i(M)=t_i(\mathcal{U}(M))$.}:
 \begin{gather*}
t_0(  \HH_k( \GL(R, I); \kk)  ) \leqslant  4k + 3,\\
t_1(  \HH_k( \GL(R, I); \kk)  ) \leqslant  4k + 8.
\end{gather*}
It is known that $\HH_k( \GL_n(R,I); \kk)$ is finite dimensional for every $n\geqslant 0$; see \cite[Remark 1.6]{cefn}. Hence, the $\FI$-module $\HH_k( \GL(R,I); \kk)$ is finitely generated. The corollary now follows from \cite[Theorem 1.3]{li}.
\end{proof}

\end{document}